\documentclass[11pt]{amsart}

\usepackage[T1]{fontenc}
\usepackage{lmodern}
\usepackage[margin=1.1in]{geometry}
\usepackage{amsmath, amssymb, amsthm, mathtools}
\usepackage{keytheorems}
\usepackage{xcolor}
\usepackage[
  activate={true,nocompatibility},
  final,
  protrusion=true,
  expansion=true,
  tracking=true,
  kerning=true,
  spacing=true,
  nopatch=footnote,
  factor=1100,
  stretch=10,
  shrink=10
]{microtype}
\microtypecontext{spacing=nonfrench}
\SetTracking{encoding={*},shape=sc}{40}
\usepackage{hyperref}
\hypersetup{colorlinks=true, linkcolor=blue!55!black, citecolor=green!45!black}

\newkeytheorem{theorem}[parent=section]
\newkeytheorem{proposition,lemma,corollary}[sibling=theorem]
\newkeytheorem{definition}[sibling=theorem,style=definition]
\newkeytheorem{remark}[sibling=theorem,style=remark]

\newcommand{\fgr}[1]{[#1\mkern1.5mu]}
\newcommand{\lfgr}[1]{[#1\mkern1.5mu]_{1}}
\newcommand{\R}{\mathbb{R}}
\newcommand{\Z}{\mathbb{Z}}
\newcommand{\LL}{\mathrm{L}}
\newcommand{\1}{\mathbf{1}}
\newcommand{\eqr}{\mathcal{R}}
\newcommand{\Aut}{\mathrm{Aut}}
\newcommand{\abs}[1]{\left\lvert #1\right\rvert}
\newcommand{\norm}[1]{\lVert #1\rVert}
\newcommand{\acts}{\curvearrowright}
\newcommand{\symdiff}{\mathbin{\vartriangle}}

\title[The analogue of Belinskaya's theorem for measure-preserving flows]{The analogue of Belinskaya's theorem \\ for measure-preserving flows}
\author{Konstantin Slutsky}
\address{Department of Mathematics, Iowa State University, Ames, IA 50011, USA}
\thanks{The author gratefully acknowledges the support of the National Science Foundation under
  grant~\texttt{DMS-2153981}.}

\begin{document}

\begin{abstract}
  We prove the analogue of Belinskaya's theorem for measure-preserving flows: two free ergodic
  measure-preserving flows whose \(\LL^{1}\) full groups are isomorphic as abstract groups are
  conjugate up to a scalar time change.  This answers a question posed by Fran\c{c}ois Le
  Ma\^{\i}tre and the author.  We show that whenever two free ergodic flows generate the same orbit
  equivalence relation and one is contained in the other's \(\LL^{1}\) full group, their positive
  half-orbits are commensurate after possibly reversing time.  Katznelson's criterion then yields
  conjugacy after a scalar time change.

  The key new ingredient is a commensuration criterion asserting that a measurable subset of the
  real line whose symmetric differences with its translates have finite average measure over the
  unit interval is commensurate with exactly one of the empty set, the whole line, and the two
  half-lines.  This criterion and its application were discovered autonomously by a two-agent AI
  system.  The author independently verified the proofs and prepared the final text.
\end{abstract}

\maketitle

\section{Introduction}
\label{sec:introduction}
A recurring theme in ergodic theory is to understand how much of a dynamical system is remembered by
its partition into orbits.  Full groups provide a natural algebraic framework for this question.  To
recall their definition, consider a Borel measure-preserving action of a Polish group on a standard
probability space \((X,\mu)\).  The full group of the action consists of all measure-preserving
transformations of \(X\) that send almost every point to a point in the same orbit.  Introduced by
Dye~\cite{Dye} in the context of actions of countable groups, full groups provide algebraic
invariants of orbit equivalence.  Orbit full groups associated with actions of general Polish
groups, and with actions of locally compact groups in particular, were studied
in~\cite{CLM-more,CLM-locally-compact}.  Since Dye's work, an important line of research has sought
to understand how dynamical properties of an action are reflected in algebraic properties of its
full group and of distinguished subgroups of it.  This viewpoint has led to variants of the
full-group construction in which one restricts how transformations are allowed to move points inside
their orbits, both within the context of ergodic theory~\cite{CJLT,LM,Rud} and in topological
dynamics~\cite{BM,GPS}.  Such restrictions can retain dynamical information that the full group
itself forgets.

The \(\LL^{1}\) full group is one such variant.  Suppose that a normed Polish group \(G\) acts freely
and by measure-preserving transformations on \((X,\mu)\).  Every element \(T\) of the full group has
a cocycle \(\rho_{T}:X\to G\), defined by \(Tx=\rho_{T}(x)\cdot x\), and the \(\LL^{1}\) full group
consists of those \(T\) for which
\[\norm{T}_{1}=\int_X\norm{\rho_T(x)}\,d\mu(x)<\infty.\]
Le Ma\^{\i}tre introduced this concept for measure-preserving transformations~\cite{LM}.  In
joint work with the author~\cite{LMS}, it was extended to Borel actions of general Polish normed
groups.  The \(\LL^{1}\) norm induces a Polish group topology, and abstractly isomorphic \(\LL^{1}\)
full groups of ergodic actions of compactly generated locally compact Polish groups determine the
same \(\LL^{1}\) orbit equivalence class~\cite[Prop.~4.21]{LMS}.  Here two actions are \(\LL^{1}\)
orbit equivalent if, after conjugating one of them, they have the same orbits and each action is
contained in the \(\LL^{1}\) full group of the other.  Thus \(\LL^{1}\) full groups carry both a
rich topological group structure and substantially finer dynamical information than ordinary full
groups.

The pivotal rigidity result in this direction is Belinskaya's theorem for \(\Z\)-actions~\cite{Bel}.
Dye's orbit-equivalence theorem~\cite{Dye} implies that the ordinary full groups of all free ergodic
measure-preserving \(\Z\)-actions are isomorphic, so the full group alone cannot recover the
generating transformation.  Belinskaya's theorem shows that the integrability restriction changes
the picture completely: two ergodic measure-preserving transformations that are \(\LL^{1}\) orbit
equivalent are flip conjugate.  Equivalently, in the full-group formulation of~\cite{LM}, if
\(T_{1},T_{2} \in \Aut(X,\mu)\) are ergodic and have isomorphic \(\LL^{1}\) full groups, then
\(T_{1}\) is conjugate to either \(T_{2}\) or \(T_{2}^{-1}\).  Hence the \(\LL^{1}\) full group
remembers an ergodic \(\Z\)-action up to a change of orientation.  The integrability assumption in
Belinskaya's theorem
was shown to be sharp by Carderi, Joseph, Le Ma\^{\i}tre, and Tessera~\cite{CJLT}.

A particularly clean route to Belinskaya's theorem is provided by a conjugacy criterion of
Katznelson; see~\cite[Appendix~A]{CJLT}.  If two aperiodic measure-preserving transformations
\(T_{1}\) and \(T_{2}\) generate the same orbit equivalence relation and their positive half-orbits
have finite symmetric difference,
\[\abs{\{T_{1}^{n}x:n\ge 0\}\symdiff\{T_{2}^{n}x:n\ge 0\}}<\infty \quad \text{for almost every \(x\),}\]
then they are conjugate by an element of their common full group.  The integrability assumption is
then converted into commensuration of the positive half-orbits, and commensuration yields
conjugacy (cf.~\cite[Cor.~3.7]{Bel}).

The work~\cite{LMS} established an analogue of Katznelson's criterion for flows.  Suppose that two
free measure-preserving flows \(\mathcal F_1\) and \(\mathcal F_2\) share their orbits and induce
the same Lebesgue measures \(\lambda_x\) on them.  Let \(s_i(x)\) denote the right half-orbit of
\(x\) under \(\mathcal F_i\).  If
\[\lambda_x\bigl(s_1(x)\symdiff s_2(x)\bigr)<\infty \quad \text{for almost every \(x\),}\]
then the two flows are conjugate~\cite[Thm.~10.9]{LMS}.  This reduces a global conjugacy problem to
a concrete commensuration problem on almost every orbit.

This criterion led to two questions in~\cite{LMS}.  The first asks whether two free ergodic
measure-preserving flows with isomorphic \(\LL^{1}\) full groups must be conjugate after rescaling
time by a nonzero scalar.  This is the strongest conclusion one can expect, since time rescaling
does not change the \(\LL^{1}\) full group.  The second asks whether two ergodic flows with the same
orbits and the same \(\LL^{1}\) full group must, after matching their orbit measures and possibly
reversing time, satisfy the hypothesis of the flow version of Katznelson's criterion.  These are
Questions~10.17 and~10.10 of~\cite{LMS}, respectively.

The present note answers both questions affirmatively.  The first main result is the following flow
version of Belinskaya's theorem~\cite[Cor.~3.7]{Bel}.  For a flow \(\mathcal F\), write
\(\eqr_{\mathcal F}\) for its orbit equivalence relation and \(\lfgr{\mathcal F}\) for its
\(\LL^{1}\) full group.  In subgroup inclusions, we identify \(\mathcal F\) with its subgroup of time
maps; thus \(\mathcal F' \le \lfgr{\mathcal F}\) means that every time map of \(\mathcal F'\) belongs
to \(\lfgr{\mathcal F}\).  Finally, \(\R^{*} = \R \setminus \{0\}\), and
\(m_{\alpha} : \R \to \R\) is the automorphism \(r \mapsto \alpha r\).

\getkeytheorem{belinskaya-flows-same-orbits}

As in the case of \(\Z\)-actions~\cite{Bel}, the hypothesis is one-sided: equality of the two
\(\LL^{1}\) full groups, or even the reverse inclusion, is not required.

Combining this result with the reconstruction of \(\LL^{1}\) orbit equivalence from abstract group
isomorphisms gives the following theorem (cf.~\cite{LM}).

\getkeytheorem{belinskaya-flows}

It was shown in~\cite[Cor.~10.16]{LMS} that ergodic flows with abstractly isomorphic \(\LL^{1}\)
full groups are flip Kakutani equivalent, meaning that they admit cross-sections whose induced
transformations are flip conjugate (see also~\cite{ORW} for background on Kakutani equivalence and
flows).  Conjugacy up to a scalar time change is considerably stronger: it recovers the entire flow,
rather than only the dynamics induced on suitable cross-sections.

The new ingredient of the present work is the commensuration criterion of
Section~\ref{sec:cut-lemma}.  It converts the \(\LL^{1}\) bound furnished by the inclusion of one
flow in the other's \(\LL^{1}\) full group into commensuration of their half-orbits, allowing the
flow version of Katznelson's criterion to be applied.

Commensuration conditions have been studied extensively, especially in the context of discrete
groups.  The general theory of commensurating group actions was developed in an excellent survey by
Cornulier~\cite{Cor}.  We also mention the characterization of property~\((T)\) in terms of
commensuration by Robertson and Steger~\cite{RS}.  In measured dynamics, Derimay recently
introduced the commensurating full group of a measure-class-preserving transformation and
established its relation to flip conjugacy~\cite{Der}.

\subsection*{Acknowledgment}
\label{sec:acknowledgment}
The author is grateful to Fran\c{c}ois Le Ma\^{\i}tre for useful comments on the earlier draft of
this work.

\subsection*{Use of AI tools}
The results of Section~2 and their application to the commensuration
question~\cite[Ques.~10.10]{LMS} via the norm formula~\cite[Prop.~6.8]{LMS} were discovered by Codex
(GPT-5.5, OpenAI) in an autonomous two-agent research loop with Claude (Opus~4.8, Anthropic).  The
agents alternately proposed, criticized, and revised candidate arguments.  The author formulated the
task, configured and initiated the loop, independently verified every mathematical step, edited the
final proofs, and assumes full responsibility for the manuscript.  Proofreading and stylistic
editing were performed with Codex (GPT-5.6-sol, OpenAI).

\section{A commensuration criterion}
\label{sec:cut-lemma}

For measurable sets \(A, B \subseteq \R\), write
\[A \symdiff B = (A \setminus B) \cup (B \setminus A)\]
for their symmetric difference.  The sets are commensurate if \(\lambda(A \symdiff B) < \infty\),
where \(\lambda\) denotes Lebesgue measure.  For \(t \in \R\), we also write \(A+t=\{a+t:a\in A\}\).
Commensuration is an equivalence relation, and all sets of finite measure form a single class.
Lemma~\ref{lem:cut} and the commensuration criterion of
Corollary~\ref{cor:commensuration-criterion}, established in this section, constitute the main new
ingredient of this note.  Apart from these two statements, the proof of
Theorem~\ref{thm:belinskaya-flows} relies only on results of~\cite{LMS} and standard measure theory.

To motivate the criterion, consider first its discrete analogue.  Let \(A \subseteq \Z\) be a set of
integers commensurate with its translate, \(\abs{A \symdiff (A + 1)} < \infty\).  Writing \(\1_A\)
for the indicator function of \(A\), the set \(A \symdiff (A + 1)\) consists of those \(k \in \Z\)
at which the indicator changes value, \(\1_{A}(k) \ne \1_{A}(k - 1)\), so its finiteness means that
the bi-infinite \(0\)--\(1\) sequence
\((\1_{A}(k))_{k \in \Z}\) changes only finitely often, hence is eventually constant in each
direction.  Writing \(\ell_{-} = \lim_{k \to -\infty} \1_{A}(k)\) and
\(\ell_{+} = \lim_{k \to +\infty} \1_{A}(k)\), the four possibilities
\((\ell_{-}, \ell_{+}) \in \{0,1\}^{2}\) correspond exactly to \(A\) being commensurate with
\[\varnothing, \qquad \Z, \qquad \Z_{\le 0} = \{k : k \le 0\}, \qquad \Z_{\ge 0} = \{k : k \ge 0\},\]
when \((\ell_{-}, \ell_{+})\) equals \((0,0), (1,1), (1,0), (0,1)\), respectively.  The criterion
below is the continuous counterpart of this example.  As there is no single distinguished unit by
which to translate in \(\R\), the role of ``finitely many changes'' is played by the finiteness of
the integral
\[\int_{0}^{1} \lambda\bigl(A \symdiff (A + t)\bigr)\, dt.\]

\begin{lemma}
  \label{lem:cut}
  Let \(A \subseteq \R\) be a measurable set satisfying
  \[I = \int_0^1 \lambda\bigl(A \symdiff (A + t)\bigr)\,dt < \infty.\]
  For \(k \in \Z\), let \(a_k = \lambda(A \cap [k,k+1))\) be the measure of \(A\) in the interval
  \([k,k+1)\).  Then
  \[\sum_{k \in \Z} a_k(1-a_k) \le I \quad \text{and} \quad
    \sum_{k \in \Z} \abs{a_{k+1}-a_k} \le 2I.\]
\end{lemma}

\begin{proof}
  For every \(t \in \R\), we have
  \begin{displaymath}
    \begin{aligned}
    \lambda\bigl(A \symdiff (A+t)\bigr) &= \int_{\R} \abs{\1_{A}(y) - \1_{A}(y-t)}\, dy \\
    [x = y-t] &= \int_{\R} \abs{\1_{A}(x) - \1_{A}(x+t)}\, dx.
    \end{aligned}
  \end{displaymath}
  Substituting
  \(x = k+u\), where \(k \in \Z\) and \(u \in [0,1)\), and applying Tonelli's theorem gives
  \[I = \int_{0}^{1}\!\int_{\R} \abs{\1_{A}(x) - \1_{A}(x+t)}\, dx\, dt = \sum_{k \in \Z}
    \int_{0}^{1}\!\int_{0}^{1} \abs{\1_{A}(k+u) - \1_{A}(k+u+t)}\, dt\, du.\]
  For \(k \in \Z\) and \(u \in [0,1)\), write \(f_{k}(u) = \1_{A}(k+u)\) for the indicator of \(A\)
  on the \(k\)-th unit interval.  For \((u,t) \in [0,1)^{2}\), the point \(k+u+t\) stays in
  \([k,k+1)\) when \(u+t < 1\), where \(\1_{A}(k+u+t) = f_{k}(u+t)\), and moves into \([k+1,k+2)\)
  when \(u+t \ge 1\), where \(\1_{A}(k+u+t) = f_{k+1}(u+t-1)\).  We split the inner integral
  accordingly.

  For the ``\(u+t < 1\)'' region,
  \[
    \begin{aligned}
      S_{k} = \iint\limits_{\mathclap{\substack{(u,t) \in [0,1)^{2}\\ u+t<1}}} \abs{f_{k}(u) - f_{k}(u+t)}\,dt\,du
      &= \int_{0}^{1}\!\int_{u}^{1} \abs{f_{k}(u) - f_{k}(v)}\,dv\,du \\
      \text{[by symmetry]} &= \frac{1}{2} \int_{0}^{1}\!\int_{0}^{1} \abs{f_{k}(u) - f_{k}(v)}\,dv\,du = a_{k}(1-a_{k}),
    \end{aligned}
  \]
  where the last equality holds because \(f_{k}\) is \(\{0,1\}\)-valued, so the integrand equals
  \(1\) on a set of measure \(2a_{k}(1-a_{k})\).  Hence
  \(\sum_{k \in \Z} a_{k}(1-a_{k}) = \sum_{k \in \Z} S_{k} \le I\).

  For the ``\(u+t \ge 1\)'' region, substituting \(v = u+t-1\),
  \[
    C_{k} = \iint\limits_{\mathclap{\substack{(u,t) \in [0,1)^{2}\\ u+t \ge 1}}} \abs{f_{k}(u) -
      f_{k+1}(u+t-1)}\,dt\,du = \int_{0}^{1}\!\int_{0}^{u} \abs{f_{k}(u) - f_{k+1}(v)}\,dv\,du.
  \]
  To bound the variation of \((a_{k})_{k}\), write
  \(a_{k+1} - a_{k} = \int_{0}^{1}\!\int_{0}^{1} \bigl(f_{k+1}(v) - f_{k}(u)\bigr)\,dv\,du\) and
  split the inner integral at \(v = u\):
  \[
    \abs{a_{k+1} - a_{k}} \le \int_{0}^{1}\!\int_{0}^{1} \abs{f_{k+1}(v) - f_{k}(u)}\,dv\,du = C_{k}
    + \int_{0}^{1}\!\int_{u}^{1} \abs{f_{k+1}(v) - f_{k}(u)}\,dv\,du.
  \]
  For the second piece, the inequality
  \[
    \abs{f_{k+1}(v) - f_{k}(u)} \le \abs{f_{k+1}(v) - f_{k+1}(u)} + \abs{f_{k+1}(u) - f_{k}(v)} +
    \abs{f_{k}(v) - f_{k}(u)},
  \]
  integrated over \(\{0 \le u \le v \le 1\}\), gives
  \[
    \int_{0}^{1}\!\int_{u}^{1} \abs{f_{k+1}(v) - f_{k}(u)}\,dv\,du \le S_{k+1} + C_{k} + S_{k}.
  \]
  Here the middle term \(\int_{0}^{1}\!\int_{u}^{1} \abs{f_{k+1}(u) - f_{k}(v)}\,dv\,du\) equals
  \(C_{k}\) after exchanging the roles of \(u\) and~\(v\).  Summing over \(k\),
  \[
    \sum_{k \in \Z} \abs{a_{k+1} - a_{k}} \le \sum_{k \in \Z} \bigl(C_{k} + S_{k+1} + C_{k} +
    S_{k}\bigr) = 2 \sum_{k \in \Z} C_{k} + 2 \sum_{k \in \Z} S_{k} = 2I. \qedhere
  \]
\end{proof}

Lemma~\ref{lem:cut} is the key technical tool behind the following commensuration criterion, which
is the statement ultimately needed in Section~\ref{sec:proof-belinskaya-flows}.
\begin{corollary}
  \label{cor:commensuration-criterion}
  If \(A \subseteq \R\) is a measurable set satisfying
  \[I = \int_0^1 \lambda\bigl(A \symdiff (A + t)\bigr)\,dt < \infty,\]
  then \(A\) is commensurate with exactly one of \(\varnothing\), \(\R\), \((-\infty,0]\), and
  \([0,\infty)\).
\end{corollary}

\begin{proof}
  For \(k \in \Z\), put \(a_{k} = \lambda\bigl(A \cap [k,k+1)\bigr) \in [0,1]\).  By Lemma~\ref{lem:cut},
  \begin{equation}
    \label{eq:cut-summable}
    \sum_{k \in \Z} a_{k}(1 - a_{k}) < \infty \qquad \text{and} \qquad
    \sum_{k \in \Z} \abs{a_{k+1} - a_{k}} < \infty.
  \end{equation}
  The finiteness of the second sum means that \((a_{k})_{k}\) has finite total variation, so its
  positive and negative tails are Cauchy.  The limits
  \[\ell_{+} = \lim_{k \to +\infty} a_{k} \qquad \text{and} \qquad
    \ell_{-} = \lim_{k \to -\infty} a_{k}\]
  therefore exist.  The finiteness of the first sum forces \(a_{k}(1 - a_{k}) \to 0\) as
  \(k \to \pm\infty\), so that \(\ell_{+}, \ell_{-} \in \{0, 1\}\).

  We show that the pair \((\ell_{-}, \ell_{+}) \in \{0,1\}^{2}\) pins down the commensuration class
  of \(A\).  Consider the positive direction.  If \(\ell_{+} = 0\), pick \(N \ge 0\) with
  \(a_{k} \le \tfrac{1}{2}\) for all \(k \ge N\).  Then \(a_{k} \le 2 a_{k}(1 - a_{k})\) for such
  \(k\), so by~\eqref{eq:cut-summable}
  \[\lambda\bigl(A \cap [N, \infty)\bigr) = \sum_{k \ge N} a_{k} \le 2 \sum_{k \ge N} a_{k}(1 - a_{k})
    < \infty,\]
  and since \([0, N)\) has finite measure, \(\lambda\bigl(A \cap [0, \infty)\bigr) < \infty\).  If
  instead \(\ell_{+} = 1\), choose \(N \ge 0\) so that \(a_{k} \ge \tfrac{1}{2}\) for all
  \(k \ge N\).  The same estimate applied to \(1 - a_{k}\) gives
  \(\sum_{k \ge N}(1 - a_{k}) < \infty\).  The finitely many intervals with \(0 \le k < N\)
  contribute finite measure, so \(\lambda\bigl([0, \infty) \setminus A\bigr) < \infty\).  The negative
  direction is identical, with \([0, \infty)\) replaced by \((-\infty, 0]\) and \(\ell_{+}\) by
  \(\ell_{-}\): the case \(\ell_{-} = 0\) gives \(\lambda\bigl(A \cap (-\infty, 0]\bigr) < \infty\),
  and the case \(\ell_{-} = 1\) gives \(\lambda\bigl((-\infty, 0] \setminus A\bigr) < \infty\).

  Combining the two directions according to the value of \((\ell_{-}, \ell_{+})\) yields finiteness
  of the symmetric difference of \(A\) with one of the four reference sets:
  \begin{align*}
    (\ell_{-}, \ell_{+}) = (0,0) &: \lambda(A \symdiff \varnothing) < \infty, &
    (\ell_{-}, \ell_{+}) = (1,0) &: \lambda\bigl(A \symdiff (-\infty,0]\bigr) < \infty, \\
    (\ell_{-}, \ell_{+}) = (1,1) &: \lambda(A \symdiff \R) < \infty, &
    (\ell_{-}, \ell_{+}) = (0,1) &: \lambda\bigl(A \symdiff [0,+\infty)\bigr) < \infty.
  \end{align*}
  As the four reference sets are pairwise noncommensurate, \(A\) is commensurate with exactly one of
  them.
\end{proof}

\section{Belinskaya's theorem for measure-preserving flows}
\label{sec:proof-belinskaya-flows}

We begin by recalling the relevant facts and notation from~\cite{LMS}.  Let \(\mathcal{F} \acts X\)
be a free measure-preserving \(\R\)-flow on a standard probability space \((X, \mu)\).  By the
classical results of Mackey and Ramsay~\cite{Mac,Ram}, we may regard \(\mathcal{F}\) as a free Borel
action of \(\R\) on a standard Borel space.  As usual, transformations are identified modulo null
sets, and equalities of actions and orbit relations are understood after restriction to invariant
conull Borel sets.  Throughout this
section, we use the additive notation \(x + r\) for the action of \(\mathcal{F}\), where \(x \in X\)
and \(r \in \R\).  Let \(\eqr\) denote its orbit equivalence relation:
\[\eqr = \{(x,y) : x + r = y \text{ for some } r \in \R\}.\]
Identifying the orbit of a point \(x\) with \(\R\) through the map \(r \mapsto x + r\), Lebesgue
measure on \(\R\) transports to a \(\sigma\)-finite measure \(\lambda_{x}\) on the orbit \([x]_{\eqr}\):
\[\lambda_{x}(B) = \lambda(\{r \in \R : x + r \in B\}).\]
These measures satisfy \(\lambda_{x} = \lambda_{y}\) whenever \(x \eqr y\) and are used to define a
\(\sigma\)-finite measure \(M\) on \(\eqr\) by
\[M(A) = \int_{X} \lambda_{x}(A_{x})\, d\mu(x), \qquad A_{x} = \{y : (x,y) \in A\}.\]
The flow linearly orders each orbit by declaring \(x \le y\) whenever \(y \in x + [0, \infty)\).  We
write \(s(x) = x + [0,\infty)\) for the right half-orbit of \(x\) and set
\[\eqr^{\le} = \{(x,y) \in \eqr : x \le y\}.\]

Write \(\fgr{\mathcal{F}}\) and \(\lfgr{\mathcal{F}}\) for the full group and the \(\LL^{1}\) full
group of \(\mathcal{F}\), respectively.  Every element \(T\) of \(\fgr{\mathcal{F}}\) has an
associated cocycle
\(\rho_{T} : X \to \R\), defined by \(Tx = x + \rho_{T}(x)\).  The \(\LL^{1}\) norm of \(T\) is
given by \(\norm{T}_{1} = \int_{X}\abs{\rho_{T}}\, d\mu\).  For \(T \in \fgr{\mathcal{F}}\), let
\(r_{T}\) be the measure-preserving transformation of \((\eqr, M)\) given by
\(r_{T}(x,y) = (x, Ty)\).  The \(\LL^{1}\) norm admits the following description~\cite[Prop.~6.8]{LMS}:
\begin{equation}
  \label{eq:norm-M}
  \norm{T}_{1} = M\bigl(\eqr^{\le} \symdiff r_{T}(\eqr^{\le})\bigr).
\end{equation}
The cited proposition uses \(\eqr^{\ge}\) instead; the formulas are equivalent by taking complements,
since the diagonal is \(M\)-null.
Since the fiber \(\bigl(\eqr^{\le} \symdiff r_{T}(\eqr^{\le})\bigr)_{x}\) is equal to
\(s(x) \symdiff Ts(x)\), Equation~\eqref{eq:norm-M} gives
\begin{equation}
  \label{eq:norm-M-half-orbit}
  \norm{T}_{1} = \int_{X} \lambda_{x}\bigl(s(x) \symdiff Ts(x)\bigr) \, d\mu(x).
\end{equation}
In particular, an element \(T \in \fgr{\mathcal{F}}\) belongs to the \(\LL^{1}\) full group
\(\lfgr{\mathcal{F}}\) precisely when the integral in~\eqref{eq:norm-M-half-orbit} is finite.

For \(\alpha \in \R^{*} = \R \setminus \{0\}\), let \(m_{\alpha}\) be the automorphism of \(\R\)
given by multiplication by \(\alpha\): \(r \mapsto \alpha r\).  The flow
\(\mathcal{F} \circ m_{\alpha}\) is given by
\((x,r) \mapsto x + \alpha r\) and is viewed as a time rescaling of \(\mathcal{F}\), including time
reversal when \(\alpha < 0\).

We will use the flow version of Katznelson's criterion~\cite[Thm.~10.9]{LMS}.  Suppose that two
free measure-preserving flows \(\mathcal{F}_{1}\) and \(\mathcal{F}_{2}\) share their orbits and
induce the same orbit measures \(\lambda_{x}\).  If their right half-orbits \(s_{1}(x)\) and
\(s_{2}(x)\) satisfy
\[\lambda_{x}\bigl(s_{1}(x) \symdiff s_{2}(x)\bigr) < \infty
  \quad \text{for \(\mu\)-almost every \(x\),}\]
then \(\mathcal{F}_{1}\) and \(\mathcal{F}_{2}\) are conjugate.

As above, we identify a flow with its subgroup of time maps when using subgroup notation.

\begin{theorem}[store=belinskaya-flows-same-orbits]
  \label{thm:belinskaya-flows-same-orbits}
  Let \(\mathcal{F}\) and \(\mathcal{F}'\) be free ergodic measure-preserving flows on \((X,\mu)\)
  that share the same orbits, \(\eqr_{\mathcal{F}} = \eqr_{\mathcal{F}'}\).  If
  \(\mathcal{F}' \le \lfgr{\mathcal{F}}\), then there exists \(\alpha \in \R^{*}\) such that
  \(\mathcal{F}'\) and \(\mathcal{F} \circ m_{\alpha}\) are conjugate.
\end{theorem}

\begin{proof}
  The roles of the flows are asymmetric, and we view the orbits from the perspective of
  \(\mathcal{F}\).  Write, as before, \(x + r\) for the action of \(\mathcal{F}\), and retain the
  convention that \(x \le y\) if \(y = x + r\) for some \(r \ge 0\).  Let
  \((\lambda_{x})_{x \in X}\) and \((\lambda'_{x})_{x \in X}\) be the orbit measures induced by
  \(\mathcal{F}\) and \(\mathcal{F}'\), respectively.  We view the time maps of \(\mathcal{F}'\) as a
  one-parameter family of measure-preserving transformations
  \(T_{t} \in \lfgr{\mathcal{F}}\), \(t \in \R\).  Since
  \(\mathcal{F}' \le \fgr{\mathcal{F}}\) and \(\R\) is unimodular, its time maps preserve
  \(\lambda_{x}\) on almost every orbit~\cite[Prop.~4.13]{LMS}.  Consequently, there is a measurable
  \(\eqr\)-invariant function \(c : X \to (0,\infty)\) such that
  \(\lambda_{x} = c(x)\lambda'_{x}\) (see the discussion preceding~\cite[Thm.~10.9]{LMS}).
  Ergodicity gives \(c(x) = c > 0\) almost everywhere.  Scalar time change does not alter the
  \(\LL^{1}\) full group as a set, while the orbit measure of
  \(\mathcal{F} \circ m_{c}\) is \(c^{-1}\lambda_{x}\).  Thus, after replacing \(\mathcal{F}\) by
  \(\mathcal{F} \circ m_{c}\) and relabeling the notation, we may assume that
  \(\lambda_{x} = \lambda'_{x}\).

  After this normalization, it suffices to prove that \(\mathcal{F}'\) is conjugate either to
  \(\mathcal{F}\) or to its time reversal \(\mathcal{F} \circ m_{-1}\).  Undoing the normalization
  will then give the result with \(\alpha = \pm c\).

  Let \(N(t)\) denote the \(\LL^{1}\) norm of
  \(T_{t}\) viewed as an element of \(\lfgr{\mathcal{F}}\).  Applying~\eqref{eq:norm-M-half-orbit}
  to \(\mathcal{F}\) and \(T_{t}\) gives
  \begin{equation}
    \label{eq:norm-F-prime}
    N(t) = \norm{T_{t}}_{1} = \int_{X} \lambda_{x}\bigl(s(x) \symdiff T_{t}
      s(x)\bigr)\, d\mu(x) < \infty.
  \end{equation}
  Equation~\eqref{eq:norm-F-prime} is expressed in terms of \(\mathcal{F}\), and \(\lambda_{x}\) is
  the orbit measure induced by \(\mathcal{F}\).  However, now that the flows have been rescaled to
  induce the same measures on their orbits, \(\lambda_{x} = \lambda'_{x}\), we can also interpret
  \(\lambda_{x}\) as the orbit measure induced by \(\mathcal{F}'\):
  \[\lambda_{x}(s(x) \symdiff T_{t}s(x)) = \lambda(\{r \in \R : T_{r}x \in s(x) \symdiff
    T_{t}s(x)\}).\]
  Consider the Borel set
  \[A = \{(x,u) \in X \times \R : (x, T_{u}x) \in \eqr^{\le}\} = \{(x,u) \in X \times \R : x \le
    T_{u}x\}.\]
  If \(A_{x} = \{u \in \R : (x,u) \in A\}\) is the fiber of \(A\) over \(x\), then
  \(T_{r}x \in T_{t}s(x)\) if and only if \(T_{r-t}x \in s(x)\), and hence
  \[\{r \in \R : T_{r}x \in s(x) \symdiff
    T_{t}s(x)\} = A_{x} \symdiff (A_{x} + t).\]
  Therefore, Equation~\eqref{eq:norm-F-prime} becomes
  \begin{equation}
    \label{eq:norm-integrand}
    N(t) = \int_{X} \lambda\bigl(A_{x} \symdiff (A_{x} + t)\bigr)\, d\mu(x).
  \end{equation}
  Since
  \[\lambda\bigl(A_{x} \symdiff (A_{x} + t)\bigr)
    = \int_{\R} \abs{\1_{A}(x,u) - \1_{A}(x, u - t)}\, du,\]
  the function \(N\) is measurable by Tonelli's theorem.

  The map \(t \mapsto T_{t}\) is a one-parameter subgroup, and so \(N\), being a norm, is subadditive,
  \[N(t + s) \le N(t) + N(s),\]
  and satisfies \(N(-t) = N(t)\).  A finite, measurable, symmetric, subadditive function on \(\R\)
  is bounded on compact sets.  Indeed, some level set \(E = \{t \in [0,1] : N(t) \le K\}\) has
  positive measure.  By Steinhaus's theorem~\cite{Ste} (see also~\cite{Str} and~\cite[Exer.~1.6.8]{Tao}), the
  difference set \(E-E\) contains an interval \((-\delta,\delta)\).  If \(u,v \in E\), then
  \[N(u-v) \le N(u) + N(-v) \le 2K.\]
  Thus \(N \le 2K\) on \((-\delta,\delta)\).  Choosing a positive integer \(n\) with \(1/n < \delta\),
  subadditivity gives
  \[N(t) \le n N(t/n) \le 2Kn \quad \text{for every } t \in [0,1].\]

  In particular, \(\int_{0}^{1} N(t)\, dt < \infty\).  Integrating~\eqref{eq:norm-integrand} over
  \(t \in [0,1]\) and applying Tonelli's theorem once more, we obtain
  \begin{equation}
    \label{eq:Ix-finite}
    I_{x} = \int_{0}^{1} \lambda\bigl(A_{x} \symdiff (A_{x} + t)\bigr)\, dt < \infty \quad \text{for
        \(\mu\)-almost every \(x\).}
  \end{equation}

  Fix \(x\) with \(I_{x} < \infty\).  By Corollary~\ref{cor:commensuration-criterion},
  \(A_{x} \subseteq \R\) is commensurate with exactly one of \(\varnothing\), \(\R\), \((-\infty,0]\),
  and \([0,\infty)\).  As \(\lambda(A_{x}) = \lambda_{x}(s(x)) = \infty\) and
  \(\lambda(\R \setminus A_{x}) = \lambda_{x}([x]_{\eqr} \setminus s(x)) = \infty\), the first two
  are excluded, so \(A_{x}\) is commensurate with \([0,\infty)\) or with \((-\infty,0]\).  Which
  alternative occurs is an \(\eqr\)-invariant property of \(x\).  Indeed, if \(y = T_{v}x\), then
  \[A_{y} + v = \{r \in \R : T_{r}x \in s(y)\}.\]
  Since \(s(y)\) differs from \(s(x)\) by a bounded half-orbit segment of finite
  \(\lambda_{x}\)-measure, \(A_{y}\) is commensurate with a translate of \(A_{x}\), hence with the
  same ray.  By ergodicity, this \(\eqr\)-invariant choice is constant \(\mu\)-almost everywhere.
  Replacing
  \(\mathcal{F}\) by \(\mathcal{F} \circ m_{-1}\) replaces \(s(x)\), up to its endpoint, by its
  complement in \([x]_{\eqr}\).  Therefore, by~\eqref{eq:Ix-finite}, after reversing
  \(\mathcal{F}\) if necessary, we may assume that
  \begin{equation}
    \label{eq:Ax-commensurate}
    \lambda\bigl(A_{x} \symdiff [0,\infty)\bigr) < \infty \quad \text{for \(\mu\)-almost every \(x\)}.
  \end{equation}

  Let, as before, \(s(x) = x + [0,\infty)\), and set \(s'(x) = \{T_{t}x : t \ge 0\}\) to be the
  right half-orbit of \(x\) under \(\mathcal{F}'\).  Then
  \[\lambda_{x}\bigl(s(x) \symdiff s'(x)\bigr)
    = \lambda\bigl(\{r \in \R : T_{r}x \in s(x) \symdiff s'(x)\}\bigr).\]
  Since \(A_{x} = \{r : T_{r}x \in s(x)\}\),
  \begin{displaymath}
    \{r \in \R : T_{r}x \in s(x) \symdiff s'(x)\} = A_{x} \symdiff [0,\infty).
  \end{displaymath}
  By~\eqref{eq:Ax-commensurate}, we conclude that
  \(\lambda_{x}\bigl(s(x) \symdiff s'(x)\bigr) < \infty\) for \(\mu\)-almost every \(x \in X\).
  The set of \(x\) for which this holds is \(\eqr\)-invariant and conull, so after discarding a
  null set, Katznelson's criterion above applies and yields a conjugacy between the normalized
  \(\mathcal{F}\) and \(\mathcal{F}'\).  Undoing the scalar time changes, \(\mathcal{F}'\) and
  \(\mathcal{F} \circ m_{\alpha}\) are conjugate for some \(\alpha \in \R^{*}\).
\end{proof}

\begin{remark}
  \label{rem:conjugacy-no-ergodicity}
  In the setting of Theorem~\ref{thm:belinskaya-flows-same-orbits}, ergodicity was used twice: to
  make the scaling function \(c(x)\) constant and to make the choice between \([0,\infty)\) and
  \((-\infty,0]\) constant.
  Without ergodicity, the measurable \(\eqr\)-invariant scaling function and orientation choice may
  vary between orbits.  The argument of
  Theorem~\ref{thm:belinskaya-flows-same-orbits} gives a measurable \(\eqr\)-invariant function
  \(\alpha : X \to \R^{*}\) such that \(\mathcal{F}'\) is conjugate to the flow
  \(\mathcal{F}_{\alpha}\) defined by \((x,r) \mapsto x + \alpha(x) r\).
\end{remark}

Finally, we note that the standard reconstruction techniques show that an abstract isomorphism of
the \(\LL^{1}\) full groups of ergodic flows yields a conjugacy of the flows up to rescaling.

\begin{theorem}[store=belinskaya-flows]
  \label{thm:belinskaya-flows}
  Let \(\mathcal{F}_{i}\) be a free ergodic measure-preserving flow on a standard probability space
  \((X_i,\mu_i)\), for \(i=1,2\).  Suppose that their \(\LL^{1}\) full groups
  \(\lfgr{\mathcal{F}_{1}}\) and \(\lfgr{\mathcal{F}_{2}}\) are isomorphic as abstract groups.  Then
  there exists \(\alpha \in \R^{*}\) such that \(\mathcal{F}_{2}\) and
  \(\mathcal{F}_{1} \circ m_{\alpha}\) are conjugate.
\end{theorem}

\begin{proof}
  By~\cite[Prop.~4.21]{LMS}, the isomorphism \(\lfgr{\mathcal{F}_{1}} \cong \lfgr{\mathcal{F}_{2}}\)
  implies that \(\mathcal{F}_{1}\) and \(\mathcal{F}_{2}\) are \(\LL^{1}\) orbit equivalent in the
  sense of~\cite[Def.~4.19]{LMS}: the isomorphism is realized spatially through Fremlin's
  reconstruction theorem~\cite[384D]{Fre}, using that \(\LL^{1}\) full groups of ergodic flows have
  many involutions~\cite[Lem.~3.7]{LMS}.  By~\cite[Cor.~4.24]{LMS}, after conjugating
  \(\mathcal{F}_{1}\) to an action on \((X_2,\mu_2)\), we may assume that the two flows share the same
  orbits,
  \(\eqr_{\mathcal{F}_{1}} = \eqr_{\mathcal{F}_{2}} = \eqr\), and that
  \(\mathcal{F}_{2} \le \lfgr{\mathcal{F}_{1}}\).  Theorem~\ref{thm:belinskaya-flows-same-orbits}
  now implies that \(\mathcal{F}_{2}\) is conjugate to
  \(\mathcal{F}_{1} \circ m_{\alpha}\) for some \(\alpha \in \R^{*}\).
\end{proof}

\end{document}